\documentclass[10pt,oneside,a4paper]{amsart}
\usepackage{amsmath,amsfonts,amsthm, amssymb}
\usepackage{array}
\usepackage[all,textures]{xy}

\theoremstyle{plain}
\newtheorem{theorem}{Theorem}[section]
\newtheorem{proposition}[theorem]{Proposition}

\theoremstyle{definition}
\newtheorem{definition}[theorem]{Definition}

\theoremstyle{remark}

\renewcommand{\lim}{\mathrm{lim}}
\newcommand{\colim}{\mathrm{colim}}

\newcommand{\Hom}{\mathrm{Hom}}

\newcommand{\BBB}{\mathfrak{b}}

\newcommand{\Mod}{\ensuremath{\mathsf{Mod}} }

\newcommand{\Pre}{\ensuremath{\mathsf{Pr}} }
\newcommand{\Sh}{\ensuremath{\mathsf{Sh}} }

\newcommand{\lra}{\longrightarrow}

\newcommand{\bbb}{\ensuremath{\mathcal{B}}}

\newcommand{\ooo}{\ensuremath{\mathcal{O}}}

\newcommand{\CC}{\ensuremath{\mathbf{C}}}
\newcommand{\DD}{\ensuremath{\mathbf{D}}}

\CompileMatrices
\SelectTips{cm}{11}

\title{A sheaf of Hochschild complexes on quasi-compact opens}
\author{Wendy Lowen$^{\ast}$}
\address{Departement DWIS\\ Vrije Universiteit Brussel\\ Pleinlaan
2\\1050 Brussel\\ Belgium}
\email[Wendy Lowen]{wlowen@vub.ac.be}
\thanks{$^{\ast}$Postdoctoral fellow FWO/CNRS. The author acknowledges the hospitality of the 
Institut de Math\'ematiques de Jussieu (IMJ) and of 
the Institut des Hautes Etudes Scientifiques 
(IHES) during her postdoctoral fellowship with 
CNRS.}

\begin{document}
\begin{abstract}
For a scheme $X$, we construct a sheaf $\CC$ of complexes on $X$ such that for every quasi-compact open $U \subset X$, $\CC(U)$ is quasi-isomorphic to the Hochschild complex of $U$ \cite{lowenvandenbergh2}. Since $\CC$ is moreover acyclic for taking sections on quasi-compact opens, we obtain a local to global spectral sequence for Hochschild cohomology if $X$ is quasi-compact.
\end{abstract}
\maketitle

\section{Introduction}
Let $X$ be a scheme over a field $k$. In \cite{lowenvandenbergh2}, the Hochschild complex $\CC(X, \ooo_X)$ of $X$ is defined to be the Hochschild complex of the abelian category $\Mod(X)$ of sheaves on $X$.
Its cohomology theory coincides with 
various notions of Hochschild cohomology of X considered in the 
literature, for example by Swan \cite{swan} and Kontsevich \cite{kontsevich}, which in the 
commutative case agree with the earlier theory of Gerstenhaber-Schack \cite{gerstenhaberschack2}.

For a basis $\BBB$ of affine opens of $X$, there is an associated $k$-linear category (also denoted by $\BBB$) and there is a quasi-isomorphism 
$$\CC(X, \ooo_X) \cong \CC(\BBB)$$
where $\CC(\BBB)$ is the Hochschild complex of the $k$-linear category $\BBB$ (\S \ref{hochscheme}). The Hochschild complexes have a considerable amount of extra structure containing in particular the cup-product and the Gerstenhaber bracket. This extra structure is important for deformation theory. It is captured by saying that the complexes are $B_{\infty}$-algebras \cite{getzlerjones, keller6}, and $\cong$ as above means the existence of an isomorphism in the homotopy category of $B_{\infty}$-algebras. Let $\ooo_{\BBB}$ be the restriction of $\ooo_X$ to the basis $\BBB$. Then $\cong$ above is reflected in the fact that there is an equivalence between the deformation theory of $\Mod(X)$ as an abelian category \cite{lowenvandenbergh1} and the deformation theory of $\ooo_{\BBB}$ as a twisted presheaf \cite{lowen4}.

If we consider the restrictions $\BBB|_U$ of $\BBB$ to open subsets $U \subset X$, we obtain a presheaf of Hochschild complexes on $X$
$$\mathbf{C}_{\BBB}: U
\longmapsto \mathbf{C}(\BBB|_U)$$
To relate the ``global'' Hochschild complex $\CC(\BBB)$ to the ``local'' Hochschild complexes $\CC(\BBB|_U)$ of certain open subsets $U \subset X$, it would be desirable for $\CC_{\BBB}$ to be a sheaf, which is preferably acyclic for taking global sections. Unfortunately, $\CC_{\BBB}$ is not even a separated presheaf with regard to finite coverings.
In this paper, we construct a \emph{sheaf} $\CC$ of $B_{\infty}$-algebras such that
\begin{enumerate}
\item $\CC(U, \ooo_U) \cong \CC(U)$ \emph{for $U$ quasi-compact open}.
\item $\CC$ is acyclic for taking quasi-compact sections, i.e. $R\Gamma(U, \CC) = \CC(U)$ \emph{for $U$ quasi-compact open}.
\end{enumerate}
For $U$ quasi-compact open, $\CC(U)$ is obtained as a colimit of complexes $\CC_{\BBB}(U)$ over a collection $\bbb(U)$ of bases of $U$ (\S \ref{presheaf}). The properties of $\CC$ depend upon the choice of a \emph{good} presheaf $\bbb$ of bases (Definition \ref{good}). 

From properties (1) and (2), we readily deduce the existence of a local to global spectral sequence
$$E^{p,q}_2 = H^p(X, \mathbf{H}^q\CC) \Rightarrow H^{p+q}\CC(X)$$
 for Hochschild cohomology for a quasi-compact scheme $X$ (Theorem \ref{spectr}).

We should remark that for a smooth separated scheme, another sheaf of $B_{\infty}$-algebras $\DD_{\mathrm{poly}}$ is considered for example by Kontsevich \cite{Kontsevich1}, Van den Bergh \cite{vandenbergh4}, Yekutieli \cite{yekutieli1}.
Let $\CC(\ooo(U))$ be the Hochschild complex of the ring $\ooo(U)$, and let $\CC_{\mathrm{poly}}(\ooo(U))$ be the subcomplex which consists of the polydifferential operators, i.e. multilinear maps $\ooo(U)^{\otimes p} \lra \ooo(U)$  which are differential operators in each argument. Then for $U$ \emph{affine} open $\DD_{\mathrm{poly}}$ satisfies $$\DD_{\mathrm{poly}}(U) \cong \CC_{\mathrm{poly}}(\ooo(U))$$
The complex $R\Gamma(X, \DD_{\mathrm{poly}})$ computes the Hochschild cohomology of $X$, but a priori does not inherit the structure of $B_{\infty}$-algebra. One way to overcome this problem is by using a fibrant resolution $\DD_{\mathrm{poly}} \lra \mathbf{F}_{\mathrm{poly}}$ in the model category of presheaves of $B_{\infty}$-algebras as defined by Hinich \cite{hinich1}. Alternatively, in \cite[Appendix B]{vandenbergh4}, Van den Bergh constructs a quasi-isomorphic object $R\Gamma(X, \DD_{\mathrm{poly}})^{\mathrm{tot}}$ that \emph{does} inherit this structure (the construction, which uses pro-hypercoverings, is functorial and inherits any operad-algebra structure).  Moreover, $R\Gamma(X, \DD_{\mathrm{poly}})^{\mathrm{tot}}$ is isomorphic to $\CC(X, \ooo_X)$ in the homotopy category of $B_{\infty}$-algebras \cite[Theorem 3.1, Appendices A, B]{vandenbergh4} and by \cite[Appendix B.10]{vandenbergh4}, we actually have $R\Gamma(X, \DD_{\mathrm{poly}})^{\mathrm{tot}} \cong \Gamma(X, \mathbf{F}_{\mathrm{poly}})$ in the same sense.

Finally, as to the existence of a local to global spectral sequence for Hochschild cohomology for a general ringed space $(X, \ooo_X)$, a proof using hypercoverings is in preparation \cite{lowenvandenbergh3}.

\section{Presheaves of Hochschild complexes}
\subsection{The Hochschild complex of a scheme}\label{hochscheme}
Throughout, $k$ is a field.
Let $(X, \ooo_X)$ be a scheme over $k$ and let $\BBB$ be a basis of affine opens.
We use the same notation for the associated $k$-linear category with $\BBB$ as objects and 
$$\BBB(V,U) = \begin{cases} \ooo_X(V) & \text{if}\, \, \, V \subset U \\ 0 & \text{else} \end{cases}$$
In \cite[\S 7.1]{lowenvandenbergh2}, the Hochschild complex $\mathbf{C}(X, \ooo_X)$ of $X$ is defined and in \cite[Theorem 7.3.1]{lowenvandenbergh2}, there is shown to be a quasi-isomorphism $$\CC(X, \ooo_X) \cong \CC(\BBB)$$
where $\CC(\BBB)$ is the Hochschild complex of the $k$-linear category $\BBB$ \cite{mitchell}, i.e.  
$$\CC^p(\BBB) = \prod_{U_0, \dots, U_p \in \BBB} \Hom_k(\BBB(U_{p-1}, U_p) \otimes_k \dots \otimes_k \BBB(U_0, U_1), \BBB(U_0, U_p))$$
and the differential is the usual Hochschild differential. More concretely, we have
$$\CC^p(\BBB) = \prod_{U_0 \subset U_1 \subset \dots \subset U_p \in \BBB} \Hom_k(\ooo_X(U_{p-1}) \otimes_k  \dots \otimes_k \ooo_X(U_0), \ooo_X(U_0))$$
$$\CC^0(\BBB) = \prod_{U_0 \in \BBB} \ooo_X(U_0)$$
Hence this complex combines the nerve of the poset $\BBB$ with the algebraic structure of $\ooo_X$.
In fact, both complexes are $B_{\infty}$-algebras \cite{getzlerjones, keller6} and $\cong$ means the existence of an isomorphism in the homotopy category of $B_{\infty}$-algebras.

\subsection{The presheaf $\CC_{\BBB}$ of Hochschild complexes}
For an arbitrary open 
subset $U \subset X$, put $\BBB|_U = \{B \in
\BBB \, |\, B \subset U\}$. Then $\BBB|_U$ is a basis of affine opens for the scheme $(U, \ooo_U)$, hence we have a quasi-isomorphism
$$\CC(U, \ooo_U) \cong \CC(\BBB|_U)$$
For $V \subset U$ there is an obvious restriction map
$$\CC(\BBB|_U) \lra \CC(\BBB|_V)$$
We thus obtain a presheaf $$\CC_{\BBB}: U \longmapsto \CC_{\BBB}(U) = \CC(\BBB|_U)$$ of Hochschild complexes on $X$.
It is readily seen that in general, $\CC_{\BBB}$ fails to be a sheaf. Indeed, suppose we have $W \in \BBB$ and $W = U \cup V$ with $U$ and $V$ proper open subsets. Then there is a non-zero element $\varphi = (\varphi_{U_0})_{U_0} \in \CC^0_{\BBB}(W)$ with 
$$\varphi_{U_0} = \begin{cases} 1\in \ooo_X(U_0) & \text{if}\, \, \, U_0 = W \\ 0 & \text{else} \end{cases}$$
whose restriction to $U$ and $V$ is zero. In this example, the fact that $W = U \cup V$ makes the presence of $W$ in $\BBB$ redundant.  
This suggests that to obtain a sheaf, we have to work with variable bases, as will be done in the next section. 

\subsection{The presheaf $\CC_{\bbb}$ of colimit Hochschild complexes}\label{presheaf}
In this section we will consider instead of
$\mathbf{C}_{\BBB}(U)$ for a fixed basis $\BBB$ of $X$, a colimit of complexes $\mathbf{C}(\BBB)$ over
different bases $\BBB$ of $U$. More precisely, we are looking for collections $\bbb(U)$ of bases of affine opens of $U$, which allow us to define ``colimit Hochschild complexes''
$$\mathbf{C}_{\bbb}(U) = \colim_{\BBB \in \bbb(U)}\mathbf{C}(\BBB)$$
Here $\bbb(U)$ is ordered by $\supset$ and $\BBB \supset \BBB'$ corresponds
to the canonical $\mathbf{C}(\BBB) \lra \mathbf{C}(\BBB')$.
Since we do not want the colimit to change the cohomology, we want it to be a
filtered colimit. In particular, this is  the case if $\bbb(U)$ is closed
under intersections, i.e. if we have the operation
\begin{equation}\label{doorsnede}
\bbb(U) \times \bbb(U) \lra \bbb(U): (\BBB,\BBB') \longmapsto \BBB \cap \BBB'
\end{equation}
Note that in general, $\BBB \cap \BBB'$ need not even be a basis.
If $\bbb(U) \neq \varnothing$ and we have (\ref{doorsnede}),
then there
are quasi-isomorphisms
$$\mathbf{C}(U,\ooo_U) \cong \mathbf{C}_{\bbb}(U)$$
For $\mathbf{C}_{\bbb}: U \longmapsto 
\mathbf{C}_{\bbb}(U)$ to become a presheaf, we need restriction operations
\begin{equation}\label{restr}
\bbb(U) \lra \bbb(V): \BBB \longmapsto \BBB|_V = \{B \in \BBB \, |\, B
\subset V\}
\end{equation}
for $V \subset U$, making $\bbb$ itself into a presheaf of collections of bases. In this way, $\CC_{\bbb}$ clearly becomes a presheaf of $B_{\infty}$-algebras on $X$.

In order to prove Proposition \ref{fincov} in the next section, we need two more operations on $\bbb$.
First, we want to take the union of bases coinciding on the intersection of
their domains, i.e. we want the operation
\begin{equation}\label{glue}
\bbb(U) \times_{\bbb(U \cap V)} \bbb(V) \lra \bbb(U\cup V): (\BBB,\BBB')
\longmapsto
\BBB
\cup
\BBB'
\end{equation}
Secondly, we want to refine bases by plugging in finer bases, i.e. for $V
\subset U$ we want the operation
\begin{equation}\label{plug}
\bbb(U) \times \bbb(V) \lra \bbb(U): (\BBB_U, \BBB_V) \longmapsto \BBB_U \circ
\BBB_V = \{B \in
\BBB_U \, |\, B \subset V \implies B \in \BBB_V\}
\end{equation}
Note that (\ref{doorsnede}) is just a special case of (\ref{plug}).
Also, combining (\ref{restr}), (\ref{glue}) and (\ref{plug}) yields the
following refinement operation on $\bbb$. If $\delta$ is any finite collection of
open subsets of $U$ (not necessarily covering $U$), we have
\begin{equation}\label{delta}
\bbb(U) \longrightarrow \bbb(U): \BBB \longmapsto \BBB_{\delta} = \{B \in
\BBB\, |\, V \subset \cup \delta \implies \exists D \in \delta, V \subset D\}
\end{equation}

\begin{definition}\label{good}
$\bbb$ is called \emph{good} if $\bbb(X) \neq \varnothing$ and $\bbb$ has the operations (\ref{doorsnede}),\dots ,
(\ref{delta}).
\end{definition}

We will now show
that there exists a good $\bbb$.
\begin{proposition}
\begin{enumerate}
\item If $\bbb$ with $\bbb(X) \neq \varnothing$ has (\ref{restr}),
(\ref{glue}) and (\ref{plug}), then it is good.
\item
Let $\BBB$ be any basis of affine opens of $X$. There exists a smallest
good $\bbb$ with $\BBB \in \bbb(X)$.
This $\bbb$ is given by
$$\bbb(U) = \{(\BBB|_U)_{\delta_1, \dots \delta_n} \, |\, \delta_i \subset
\mathrm{open}(U),\, |\delta_i| < \infty \}.$$
\end{enumerate}
\begin{proof}
(1) follows from the discussion above. For (2), first note that $\bbb$ is obviously
contained in any good $\bbb'$. 
If $V \subset U$ and $\delta$ is a collection in $U$, we put $\delta|_V
= \{D \cap V\,|\,D \in \delta\}$. For any basis $\BBB'$ of $U$, we have $(\BBB'_{\delta})|_V =
(\BBB'|_V)_{(\delta|_V)}$ so (\ref{restr}) holds.
For (\ref{plug}), note that $(\BBB|_U)_{\delta_1, \dots \delta_n} \circ
(\BBB|_V)_{\varepsilon_1, \dots \varepsilon_m} = (B|_U)_{\delta_1, \dots
\delta_n,\varepsilon_1, \dots \varepsilon_m}$. Finally for (\ref{glue}), if
$(\BBB|_U)_{\delta_1, \dots \delta_n}$ and $(\BBB|_V)_{\varepsilon_1, \dots
\varepsilon_m}$ coincide on $U \cap V$, then their union equals $(\BBB|_{U \cup
V})_{\delta_1, \dots
\delta_n,\varepsilon_1, \dots \varepsilon_m, \{U,V\}}$.
\end{proof}
\end{proposition}

\section{Sheaves of Hochschild complexes}

\subsection{The presheaf $\CC_{\bbb}$ for a good $\bbb$}
From now on, $\bbb$ is a good presheaf of bases and we consider the presheaf $\CC_{\bbb}$ of colimit Hochschild complexes as defined in \S \ref{presheaf}.

\begin{proposition}\label{fincov}
\begin{enumerate}
\item $\mathbf{C}_{\bbb}$ is flabby.
\item $\mathbf{C}_{\bbb}$ satisfies the sheaf condition with respect to finite coverings.
\end{enumerate}
\begin{proof}
(2) By induction, we may consider $U = U_1 \cup U_2$ and given elements
$\varphi_i
\in
\mathbf{C}_{\bbb}^{p}(U_i)$ such that $\varphi_1|_{U_{12}} = \varphi_2|_{U_{12}}$,
where
$U_{12} = U_1 \cap U_2$. Let $\varphi_i$ be a representing element in
$\mathbf{C}^p(\BBB_i)$ for a basis $\BBB_i \in \bbb(U_i)$ and let $\BBB'
\subset
\BBB_1|_{U_{12}} \cap \BBB_2|_{U_{12}}$ be a basis in $\bbb(U_{12})$ for
which
$\varphi_1|_{U_{12}}$ and
$\varphi_2|_{U_{12}}$ coincide in $\mathbf{C}^p(\BBB')$.
Put $\BBB_i' = \BBB_i \circ \BBB' \in \bbb(U_i)$ (using (\ref{plug})) and put 
$\BBB =
\BBB_1' \cup \BBB_2' \in \bbb(U \cup V)$ (using (\ref{glue})). We can now easily
give an element $\varphi \in \mathbf{C}^p(\BBB)$ which represents
a glueing of $\varphi_1$ and $\varphi_2$ on $U$, by specifying its value for
$V_0 \subset 
\dots \subset V_p$: if $V_p \in \BBB'_i$, we use the
element specified by $\varphi_i$. This is well defined since $V_p \in \BBB'_1
\cap \BBB'_2$ implies $V_p \in \BBB'$. It is a glueing of the $\varphi_i$ since
$\varphi$ and $\varphi_i$ coincide on $\BBB_i' \subset \BBB_i$.

Now suppose we have an element $\varphi \in
\mathbf{C}^p(\BBB')$ for $\BBB' \in \bbb(U)$ and suppose we have bases
$\BBB_i
\subset
\BBB'|_{U_i}$ for which $\varphi|_{U_i}$ becomes zero in
$\mathbf{C}^p(\BBB_i)$.
If we put $\BBB'_i = \BBB_i \circ (\BBB_1|_{U_{12}} \cap
\BBB_2|_{U_{12}})$ and $\BBB = \BBB'_1 \cup \BBB'_2$, then $\varphi$
becomes zero in $\mathbf{C}^p(\BBB')$.

(1) Consider the restriction map
$\mathbf{C}^p(U) \lra \mathbf{C}^p(V)$ for $V
\subset U$. If $\varphi \in \mathbf{C}^p(\BBB)$ is a representing
element in the codomain, we can lift it to $\bar{\varphi} \in
\mathbf{C}^p(\BBB' \circ \BBB)$ where $\BBB' \in \bbb(U)$ is arbitrary
and the value of $\bar{\varphi}$ for
$V_0 \subset \dots \subset V_p$ is the value specified by $\varphi$ if $V_p
\subset V$ and is zero else. 
\end{proof}
\end{proposition}

\subsection{The sheaf $\CC_{\mathrm{qc}}$ of colimit Hochschild complexes}

Let $\mathrm{qc}(X) \subset \mathrm{open}(X)$ be the subposet of quasi-compact opens with the
induced Grothendieck topology. We immediately get:

\begin{proposition}\label{sheaf}
The restriction $\CC_{\mathrm{qc}}$ of $\CC_{\bbb}$ to $\mathrm{qc}(X)$ is a sheaf.
\end{proposition}

\begin{proof}
Since every covering of a quasi-compact $U \subset X$ has a finite
subcovering, it suffices to check the sheaf condition on finite coverings,
which is done in Proposition \ref{fincov}.
\end{proof}

\subsection{The sheaf $\CC = a\CC_{\bbb}$}\label{thesheaf}
Let $\Pre(X)$ and $\Sh(X)$ (resp. $\Pre_{\mathrm{qc}}(X)$ and $\Sh_{\mathrm{qc}}(X)$) be the categories of presheaves and sheaves on $X$ (resp. on $\mathrm{qc}(X)$). Since $\mathrm{qc}(X)$ is a basis of $X$, by the (proof of the) Lemme the Comparaison \cite{artingrothendieckverdier1}, there is a commutative square
$$\xymatrix{\Pre(X) \ar[d]_{a} \ar[r]^-{(-)_{\mathrm{qc}}} & \Pre_{\mathrm{qc}}(X) \ar[d]^{a'}\\
\Sh(X) \ar[r]^-{(-)_{\mathrm{qc}}}_-{\cong} & \Sh_{\mathrm{qc}}(X)}$$
in which the vertical arrows are sheafifications, the horizontal arrows are restrictions to $\mathrm{qc}(X)$, and the lower horizontal arrow is an equivalence. Let $\CC = a\mathbf{C}_{\bbb}$ be the sheafification of
$\mathbf{C}_{\bbb}$.

\begin{proposition}\label{nochange}
If $U \subset X$ is a quasi-compact open, then
$$\mathbf{C}_{\bbb}(U) \lra \mathbf{C}(U)$$ is an isomorphism. In particular, there is a quasi-isomorphism
$$\CC(U, \ooo_U) \cong \CC(U)$$
\end{proposition}

\begin{proof}
By Proposition \ref{sheaf} we have $(\mathbf{C}_{\bbb})_{\mathrm{qc}} \cong a'(\mathbf{C}_{\bbb})_{\mathrm{qc}} \cong (a\mathbf{C}_{\bbb})_{\mathrm{qc}}$.
\end{proof}

\begin{proposition}\label{acyc}
$\CC^p$ is acyclic for taking quasi-compact sections, i.e. for $U \subset X$ a quasi-compact open and $i > 0$, we have $H^i(U, \CC^p) = 0$.
\end{proposition}

\begin{proof}
By Propositions \ref{fincov}(1) and \ref{nochange}, the restriction maps $\CC^p(X) \lra \CC^p(U)$ are surjective for $U$ quasi-compact. The rest of the proof is along the lines of the classical proof that flabby sheaves are acyclic for taking global sections.
\end{proof}

\section{Local to global spectral sequence}

In this section, $X$ is a quasi-compact scheme and $\CC$ is the sheaf of complexes of \S \ref{thesheaf}. In particular, there are quasi-isomorphisms $\CC(U, \ooo_U) \cong \CC(U)$ for $U$ quasi-compact open. We obtain a local to global spectral sequence for Hochschild cohomology:

\begin{theorem}\label{spectr}
There is a local to global spectral sequence $$E^{p,q}_2 = H^p(X, \mathbf{H}^q\CC) \Rightarrow H^{p+q}\CC(X)$$
\end{theorem}

\begin{proof}
Since, by Proposition \ref{acyc}, $\CC$ is a bounded below complex of acyclic objects for $\Gamma$, $\CC$ is itself acyclic i.e. $R\Gamma(X, \CC) = \CC(X)$. So the above is just the hypercohomology spectral sequence \cite[2.4.2]{grothendieck} for the complex of sheaves $\CC$.
\end{proof}

\def\cprime{$'$}
\providecommand{\bysame}{\leavevmode\hbox to3em{\hrulefill}\thinspace}
\bibliography{Bibfile}
\bibliographystyle{amsplain}
\end{document}